\documentclass[12pt,reqno]{amsart}
\usepackage{amssymb, amsmath,amsthm}
\usepackage{graphicx}
\usepackage{caption}
\newtheorem{thm}{Theorem}

\newtheorem{Example}{Example}
\numberwithin{defn}{section}
\numberwithin{thm}{section}
\numberwithin{Lemma}{section}
\numberwithin{Corollary}{section}
\numberwithin{Example}{section}
\numberwithin{subsection}{section}
\numberwithin{Remark}{section}
\numberwithin{equation}{section}
\numberwithin{ppn}{section}
\begin{document}
\title[ Third and Fourth order iterative methods ... ]
{ Third and Fourth-order iterative methods for solving systems of nonlinear equations} 
\author{Anuradha Singh and J. P. Jaiswal  }
\date{}
\maketitle


\textbf{Abstract.} The object of the present paper is to extend the third-order iterative method for solving nonlinear equations into systems of nonlinear equations. Since our motive is to develop the method which improve the order of convergence of Newton's method with minimum number of functional evaluations.To achieve this goal ,we have used weighted method. Computational efficiency is compared not only traditional way but also recently introduced  flops-like based concept. Finally numerical results are given to confirm theoretical order of convergence.
\\

\textbf{Mathematics Subject Classification (2000).}
 65H10, 65Y20.\\
\\
\textbf{Keywords and Phrases.} 
Nonlinear system, order of convergence, efficiency index, Jacobian matrix, LU factorization.
														
\section{Introduction}
 Solving the system  $F(x)=0$ of nonlinear equations is a common important problem in science and engineering \cite{Kelley}. A large number of real-world application are reduced to solve a system of nonlinear equations numerically. Solving such systems has become one of the most challenging problems in numerical analysis.
There are many approaches to solve such systems. 
In term of computational point of view when dealing with large-scale system arising from discretization of non linear PDE's \cite{Soleymani} , integral equations [\cite{Ortega}, \cite{Grau}], nonlinear boundary value problems [\cite{Barden}, \cite{Noor}],  solving the system of nonlinear equations by a direct method such as $LU$ decomposition is not easy . To solve such nonlinear systems one of the old method is Newton's method which can be written as follows:
\begin{equation*}
X^{(k+1)}=X^{(k)}-[F'(X^{(k)})]^{-1} F(X^{(k)}),
\end{equation*}
where $ [F'(X^{(k)})]^{-1} $ is the inverse of Frechet derivative  $ [F'(X^{(k)})] $  of the function $ F(X) $.The method converge quadratically, provided initial guess is close to the exact solution. In order to improve the order of convergence of the Newton's method several methods have been proposed in the literature [ \cite{Noor1}, \cite{Babajee}, \cite {Montazeri1}, \cite{Sharma}, \cite{Abad}, \cite{Sharma1}, \cite{Noor}, \cite{Grau}] and the references therein.\\
This paper is organized as follows : In section 2, first we extend the Chun third-order method \cite{Chun}  of single variable to multivariate case. After that to find higher order  (fourth-order)  convergence we use weighted method without using any more functional evaluations. The computation efficiency in more generalized form is compared with some well known recently established methods in section 3 . In section 4, the theoretical result is supported by  numerical examples . Finally we give the concluding remarks.

\section{Development of the method and convergence analysis}

By using circle of curvature concept Chun et. al. in \cite{Chun} constructed a third-order iterative methods defined by

\begin{eqnarray}\label{eqn:21}
y_n&=&x_n-\frac{f(x_n)}{f'(x_n)},\nonumber\\
x_{n+1}&=&x_n-\frac{1}{2}\left[3-\frac{f'(y_n)}{f'(x_n)}\right]\frac{f(x_n)}{f'(x_n)}.
\end{eqnarray}
Here we extend this method to the multivariate case. This method for system of nonlinear equations can be given by
\begin{eqnarray}\label{eqn:22}
Y^{(k)}&=&X^{(k)}-[F'(X^{(k)})]^{-1}F(X^{(k)}),\nonumber\\
X^{(k+1)}&=&X^{(k)}-\frac{1}{2}\left( 3I-[F'(X^{(k)})]^{-1}F'(Y^{(k)})\right) \nonumber\\
&& \times[F'(X^{(k)})]^{-1} F(X^{(k)}),
\end{eqnarray}
where $X^{(i)}=[x_1^{(i)},x_2^{(i)}, . . . ,x_n^{(i)}]^{T}$, $(i=0,1, 2, . . . )$; similarly $Y^{(i)}$; $I$ is $n \times n$ identity matrix; $F(X^{(i)})=[f_1(x_1^{(i)},x_2^{(i)}, . . . , x_n^{(i)}), f_2(x_1^{(i)}, x_2^{(i)}, . . . , x_n^{(i)}), . . . , \\ f_n(x_1^{(i)},x_2^{(i)}, . . . , x_n^{(i)})]$; and $F'(X^{(i)})$ is the Jacobian matrix of $F$ at $X^{(i)}$.

Let $\alpha+H \in \Re^n$ be the any point of the neighborhood of exact solution $\alpha \in \Re^n$ of the nonlinear system $F(X)=0$.  If Jacobian matrix $F'(\alpha)$ is nonsingular then Taylor's series expansion for multivariate case is 
\begin{eqnarray}\label{eqn:23}
F(\alpha+H)=F'(\alpha)\left[H+C_2H^2+C_3H^3+ . . . +C_{p-1}H^{p-1}\right]+O(H^p),
\end{eqnarray}

where $C_i=[F'(\alpha)]^{-1}\frac{F^{(i)}(\alpha)}{i!}$, $i\geq 2$ and 

\begin{eqnarray}\label{eqn:24}
F'(\alpha+H)&=&F'(\alpha)\left[I+2C_2H+3C_3H^2+ . . . +(p-1)C_{p-1}H^{p-2}\right] \nonumber\\
&&+O(H^{p-1}) ,
\end{eqnarray}
where $I$ is the identity matrix. From the above equation we can find
\begin{eqnarray}\label{eqn:25}
[F'(\alpha+H)]^{-1}&=&\left[I+A_1H+A_2H^2+ . . . +A_{p-1}H^{p-2}\right][F'(\alpha)]^{-1} \nonumber\\
&&+O(H^{p-1}),
\end{eqnarray}
where $A_1=-2C_2$, $A_2=4C_2^2-3C_3$, $A_3=-8C_2^3+6C_2C_3+6C_3c_2-4C_4$, ... . Here we denote the error at $k^{th}$ iterate  by  $E^{(k)}$ i.e. $E^{(k)}=X^{(k)}-\alpha$. Now the order of convergence of the method $(\ref{eqn:22})$ is confirm by the following theorem:

\begin{thm}
Let $F: D\subseteq \Re^n\rightarrow \Re^n$ be sufficiently Frechet differentiable in a convex set $D$ containing a root $\alpha$ of $F(X)=0$. Let us suppose that $F'(X)$ is continuous and nonsingular in $D$ and $X^{(0)}$ is close to $\alpha$. Then the sequence $\{X^{(0)}\}_{k\geq 0}$ obtained by the iterative expression $(\ref{eqn:22})$ converges to $\alpha$  with order three.
\end{thm} 

\begin{proof}
From  $(\ref{eqn:23})$,  $(\ref{eqn:24})$ and  $(\ref{eqn:25})$, we have

\begin{eqnarray}\label{eqn:26}
F(X^{(k)})&=&F'(\alpha)\left[E^{(k)}+C_2{E^{(k)}}^2+C_3{E^{(k)}}^3+C_4{E^{(k)}}^4+C_5{E^{(k)}}^5\right] \nonumber\\
&&+O({E^{(k)}}^6).
\end{eqnarray}
\begin{eqnarray}\label{eqn:27}
F'(X^{(k)})&=&F'(\alpha)\left[I+2C_2{E^{(k)}}+3C_3{E^{(k)}}^2+4C_4{E^{(k)}}^3+5C_5{E^{(k)}}^4\right] \nonumber\\
&&+O({E^{(k)}}^5),
\end{eqnarray}
and
\begin{eqnarray}\label{eqn:28}
&&[F'(X^{(k)})]^{-1}\nonumber\\
&&=\{I-2C_2{E^{(k)}}+(4C_2^2-3C_3){E^{(k)}}^2+(-8C_2^3+6C_2C_3+6C_3C_2 \nonumber\\
&& -4C_4){E^{(k)}}^3 \} [F'(\alpha)]^{-1}+O({E^{(k)}}^4),
\end{eqnarray}
where $C_i=[F'(\alpha)]^{-1}\frac{F^{(i)}(\alpha)}{i!}$, $i=2, 3, . . .$ .
Now from the equations $(\ref{eqn:28})$ and $(\ref{eqn:26})$, we can obtain\\
\begin{eqnarray}\label{eqn:29}
[F'(X^{(k)})]^{-1}F(X^{(k)})=E^{(k)}-C_2{E^{(k)}}^2+(2C_2^2-2C_3){E^{(k)}}^3+O({E^{(k)}}^4).
\end{eqnarray} 
By virtue of  $(\ref{eqn:29})$ the first step of the method $(\ref{eqn:22})$ becomes
\begin{eqnarray}\label{eqn:210}
Y^{(k)}=\alpha+C_2{E^{(k)}}^2+(-2C_2^2+2C_3){E^{(k)}}^3+O({E^{(k)}}^4).
\end{eqnarray} 
Now the Taylor expansion for Jacobian matrix $F'(Y^{(k)})$ can be given as
\begin{eqnarray}\label{eqn:211}
F'(Y^{(k)})=F'(\alpha)[I+2C_2^2{E^{(k)}}^2-(4C_2^4-4C_2C_3){E^{(k)}}^3]+O({E^{(k)}}^4).
\end{eqnarray} 
Therefore
\begin{eqnarray}\label{eqn:212}
&&[F'(X^{(k)})]^{-1}F'(Y^{(k)})\nonumber\\
&&=I-2C_2E^{(k)}+(6C_2^2-3C_3){E^{(k)}}^2+(-4C_2^4-12C_2^3+10C_2C_3\nonumber\\
&&+6C_3C_2-4C_4){E^{(k)}}^3+O({E^{(k)}}^4).
\end{eqnarray}
Finally using $(\ref{eqn:212})$ and $(\ref{eqn:29})$ in the second step of $(\ref{eqn:22})$, we find that the error expression can be expressed as
\begin{eqnarray}\label{eqn:213}
E^{(k+1)}=\left(2C_2^2+\frac{C_3}{3}\right){E^{(k)}}^3+O({E^{(k)}}^4),
\end{eqnarray}
which shows the theorem.
\end{proof}
 Now since our motive is to construct the method which accelerates the order of convergence of Newton's method with minimum number functional evaluations. Keeping in the mind this fact we consider the following iterative method in order to find higher order iterative method without using any additional functional evaluations.
\begin{eqnarray}\label{eqn:214}
Y^{(k)}&=&X^{(k)}-\beta [F'(X^{(k)})]^{-1}F(X^{(k)})],\nonumber\\
X^{(k+1)}&=&X^{(k)}-\frac{1}{2} \left[ 3I-[F'(X^{(k)})]^{-1}F'(Y^{(k)})\right] \nonumber\\
&& . \left[a_1I+a_2[F'(X^{(k)})]^{-1}F'(Y^{(k)})+a_3([F'(X^{(k)})]^{-1}F'(Y^{(k)}))^{2}\right]\nonumber\\
&& . [F'(X^{(k)})]^{-1}F(X^{(k)}) ],
 \end{eqnarray}
where $ \beta, a_1 , a_2,  a_3, $ are real parameters. Now the following theorem indicates for what values of $ \beta, a_1 , a_2,  a_3, $  this method achieve fourth-order convergence:
\begin{thm}
Let $F: D\subseteq \Re^n\rightarrow \Re^n$ be sufficiently Frechet differentiable in a convex set $D$ containing a root $\alpha$ of $F(X)=0$. Let us suppose that $F'(X)$ is continuous and nonsingular in $D$ and $X^{(0)}$ is close to $\alpha$. Then the sequence $\{X^{(0)}\}_{k\geq 0}$ obtained by the iterative expression $(\ref{eqn:214})$ converges to $\alpha$  with order four for $\beta=\frac{2}{3}$, $a_1=\frac{9}{4}$, $a_2=-\frac{9}{4}$ and $a_3=1$.
\end{thm} 
\begin{proof}
From $(\ref{eqn:25})$ we have
\begin{eqnarray}\label{eqn:215}
[F'(X^{(k)})]^{-1}&=&\{I-2C_2{E^{(k)}}+(4C_2^2-3C_3){E^{(k)}}^2 \nonumber\\
&&+(-8C_2^3+6C_2C_3+6C_3C_2-4C_4){E^{(k)}}^3\}[F'(\alpha)]^{-1}\nonumber\\
&&+O({E^{(k)}}^4).
\end{eqnarray}
Now multiplying $(\ref{eqn:215})$ to $(\ref{eqn:26})$, we can get
\begin{eqnarray}\label{eqn:216}
s=[F'(X^{(k)})]^{-1}F(X^{(k)})&=&E^{(k)}-C_2{E^{(k)}}^2+(2C_2^2-2C_3){E^{(k)}}^3\nonumber\\
                             &&+(-4C_2^3+4C_2C_3+3C_3C_2-3C_4){E^{(k)}}^4\nonumber\\
                             &&+O({E^{(k)}}^5).
\end{eqnarray} 
Substituting value of $(\ref{eqn:216})$ in the first step of $(\ref{eqn:214})$, we find
\begin{eqnarray}\label{eqn:217}
Y^{(k)}&=&(1-\beta)E^{(k)}+ \beta C_2{E^{(k)}}^2+ \beta(-2C_2^2+2C_3){E^{(k)}}^3\nonumber\\
       &&+ \beta (4C_2^3-4C_2C_3-3C_3C_2+3C_4) {E^{(k)}}^4\nonumber\\
       &&+O({E^{(k)}}^5).
\end{eqnarray}
By using the equation $(\ref{eqn:217})$, the Taylor expansion of Jacobian matrix $F'(Y^{(k)})$ can be written as
\begin{eqnarray}\label{eqn:218}
F'(Y^{(k)})&=&F'(\alpha)\{I+2C_2(1-\beta)E^{(k)}+(2 \beta C_2^2+3C_3(1-\beta)^2){E^{(k)}}^2\nonumber\\
        &&+(\beta(-4C_2^3+4C_2C_3)+6\beta(1-\beta)C_3C_2+4C_4(1- \beta)^3){E^{(k)}}^3\nonumber\\
       &&+(5C_5(-1+\beta)^4+\beta(8C_2^4-12C_3^2(-1+\beta))\nonumber\\
       && +C_2^2C_3(-26+15\beta)+6C_2C_4(3+2(-2+\beta)\beta)){E^{(k)}}^4\}\nonumber\\
       &&+O({E^{(k)}}^5).
\end{eqnarray}
From $(\ref{eqn:215})$ and $(\ref{eqn:218})$, it is obtained that  
\begin{eqnarray}\label{eqn:219}
 t=[F'(X^{(k)})]^{-1}F(Y^{(k)})]&= &I-2\beta C_2 E^{(k)}+\{6\beta C_2^2+3C_3(\beta^2-2\beta)\}{E^{(k)}}^2\nonumber\\
&&+ \{-16 \beta C_2^3+(-6\beta^2+16\beta)C_2C_3+6\beta(2-\beta)\nonumber\\
&& C_3C_2+(4(1-\beta)^3-4)C_4\}{E^{(k)}}^3+O({E^{(k)}}^4).  \nonumber\\
\end{eqnarray}
The above equation implies that 
\begin{eqnarray}\label{eqn:220}
&&t^2=([F'(X^{(k)})]^{-1}F(Y^{(k)}))^2\nonumber\\
&&=I-4 \beta C_2 E^{(k)}+\{(12\beta+4 \beta^2)C_2^2+6(\beta^2-2 \beta)C_3\}{E^{(k)}}^2\nonumber\\
                              &&\{(-32 \beta-24\beta^2)C_2^3+(-6\beta^3+32\beta)C_2C_3+(-6\beta^3+24\beta)C_3C_2 \nonumber\\
&&+2(4(1- \beta)^3-4)\}{E^{(k)}}^3+O({E^{(k)}}^4).
\end{eqnarray}
From  $(\ref{eqn:219})$ , we have
\begin{eqnarray}\label{eqn:220a}
3I-t&=&2I+2\beta C_2 E^{(k)}+{(-6\beta C_2^2-3(\beta^2-2\beta)C_3)}{E^{(k)}}^2\nonumber\\ 
&& \{(16\beta C_2^3+(6\beta^2-16\beta))C_2C_3+(-6\beta(2-\beta))C_3C_2\nonumber\\
&&+(-4(1-\beta)^3+4)C_4)\}{E^{(k)}}^3+O({E^{(k)}}^4).
\end{eqnarray}
Now by virtue of  $(\ref{eqn:219})$ and  $(\ref{eqn:220})$ , it can be written as
\begin{eqnarray}\label{eqn:220b}
&&a_1I+a_2t+a_3t^2 \nonumber\\
&&=(a_1+a_2+a_3)I+(-2\beta a_2-4 \beta a_3)C_2{E^{(k)}} +6( \beta^2-2 \beta )a_3)C_3]{E^{(k)}}^2\nonumber\\
&&[ (6\beta a_2+(12\beta+4\beta^2)a_3)C_2^2+(3(\beta^2-2\beta)a_2+(-16 \beta a_2+(-32 \beta\nonumber\\
&& -24 \beta^2)a_3)  C_2^3+((-6\beta^2+16\beta)a_2+(-6\beta^3+32\beta)a_3)C_2C_3+(6\beta(2-\beta)a_2\nonumber\\
&&+(-6\beta^3+24\beta)a_3)C_3C_2+((4(1-\beta)^3-4)a_2+2(4(1-\beta)^3-4)a_3)\nonumber\\
&&C_4)]{E^{(k)}}^3+O({E^{(k)}}^4).
\end{eqnarray}
Multiplying  $(\ref{eqn:220a})$ to  $(\ref{eqn:220b})$ , we have
\begin{eqnarray}\label{eqn:220c}
(3I-t)(a_1I+a_2t+a_3t^2)&= &2(a_1+a_2+a_3)I+(P_1 C_2){E^{(k)}} \nonumber\\
&&+(P_2 C_2^2+P_3 C_3){E^{(k)}}^2 +(O_1C_2^3+O_2C_2C_3  \nonumber\\                                  
&&+O_3C_3C_2+O_4C_4){E^{(k)}}^3+O({E^{(k)}}^4),
\end{eqnarray}
where
\begin{eqnarray*}
P_1 &=& 2 \beta (a_1+a_2+a_3)+2(-2\beta a_2 -4 \beta a_3) ,\nonumber\\
 P_2&=&12 \beta a_2+(24 \beta +8 \beta^2)a_3+2 \beta(-2 \beta a_2-4 \beta a_3-6 \beta(a_1+a_2+a_3)) ,\nonumber\\
 P_3&=&6(\beta^2-2\beta)a_2+12(\beta^2-2\beta)a_3-3( \beta^2 -2 \beta)(a_1+a_2+a_3),\nonumber\\
O_1&=& [-32 \beta a_2+2(-32\beta-24\beta^2)a_3+2\beta(6\beta a_2+(12 \beta+4\beta^2)a_3) \nonumber\\
&&+(-6\beta(-2\beta a_2-4\beta a_3)+16 \beta(a_1+a_2+a_3))]C_2^3, \nonumber\\
O_2&=& [2(-6 \beta^2 +16 \beta)a_2+2(-6\beta^3+32 \beta)a_3+2\beta(3(\beta^2-2\beta)a_2\nonumber\\
&& +6(\beta^2-2\beta)a_3)+(a_1+a_2+a_3)(6\beta^2-16\beta)]C_2C_3, \nonumber\\
O_3&=& [12 \beta(2-\beta)a_2+2(-6 \beta^2+24 \beta)a_3+(-3)(\beta^2-2\beta)\nonumber\\
&&(-2\beta  a_2-4\beta a_3) +(a_1+a_2+a_3)(-6 \beta (2-\beta))] C_3C_2, \nonumber\\
O_4&=&[ 2(4(1-\beta)^3-4)a_2+4(4(1-\beta)^3-4)a_3+(a_1+a_2+a_3)  \nonumber\\
&&(-4(1-\beta)^3+4)]C_4.
\end{eqnarray*} 
Again by multiplying $(\ref{eqn:220b})$ and $(\ref{eqn:216})$ , it is obtained that
\begin{eqnarray}\label{eqn:220f}
&&(3I-t)(a_1 I+a_2t+a_3t^2)s = \nonumber\\
&& 2(a_1+a_2+a_3){E^{(k)}}+(T_1C_2){E^{(k)}}^2 +(T_2C_2^2+T_3C_3){E^{(k)}}^3\nonumber\\
&&+( T_4C_2^3+T5C_2 C_3+T_6 C_3C_2+T_7C_4) {E^{(k)}}^4+O({E^{(k)}}^5),
\end{eqnarray}
where
\begin{eqnarray*}
T_1&=& P_1-2(a_1+a_2+a_3), \nonumber\\
T_2&=& P_2-P_1+4(a_1+a_2+a_3), \nonumber\\
T_3&=& P_3-4(a_1+a_2+a_3), \nonumber\\
T_4&= & O_1 -8(a_1+a_2+a_3)+2 P_1-P_2, \nonumber\\
T_5&=& O_2+8(a_1+a_2+a_3)-2 P_1, \nonumber\\
T_6&=& O_3+6(a_1+a_2+a_3)-P_3, \nonumber\\
T_7&=& O_4-6(a_1+a_2+a_3).
\end{eqnarray*}

Finally using $(\ref{eqn:220f})$, 
in the second step of method $(\ref{eqn:214})$ we get
\begin{eqnarray}\label{eqn:221}
E^{(k+1)}&&=(1-a_1-a_2-a_3)E^{(k)}- \frac{T_1}{2}{E^{(k)}}^2- \frac{1}{2}(T_2 C_2^2+T_3C_3){E^{(k)}}^3\nonumber\\
 && -\frac{1}{2}(T_4C_2^3+T_5C_2C_3+T_6C_3C_2+T_7C_4) {E^{(k)}}^4 +O({E^{(k)}}^5).
\end{eqnarray}
In order to achieve fourth-order convergence $1-a_1-a_2-a_3$, $T_1$, $T_2$ and $T_3$ must be zero. It can be easily shown that they are zero when $\beta=\frac{2}{3}$, $a_1=\frac{9}{4}$, $a_2=-\frac{9}{4}$ and $a_3=1$. 
This completes the proof .

\end{proof}

                              
\section{Efficiency Index}
The classical efficiency index of an iterative method is given by $E=p^{1/C}$, where $p$ is the order of convergence and C is the total computational cost per iteration in terms of number of functional evaluations. So here first we think about the number of functional evaluations to obtain the classical efficiency index. The computational cost for scalar function $F(X)$ is $n$ (number of scalar function evaluations) and computational cost for Jacobian $F'(X)$ is    $n^2$(number of scalar function evaluations).


Very recently it is pointed out in \cite{Montazeri1}, for the systems of nonlinear equations number of evaluations of scalar function is not only effecting factor for evaluating efficiency index. The number of scalar products, matrix products, decomposition LU of the first derivative and the resolution of the triangular linear systems are also of great importance in calculating the real efficiency index of such methods. Here we count the number of matrix quotients, products, summations and subtraction along with the cost of solving two triangular systems, that is based on flops (the real cost of solving systems of linear equations). In this case it is noted that, the flops obtaining LU decomposition is $2n^3/3$, and for solving two triangular system $2n^2$. If the right hand side is a matrix, then the cost (flops)  of the two triangular system is $2n^3$ or roughly $n^3$ as considered here. We have provided the comparison of the traditional and flops-likel efficiency indices for  (fourth-order) method (SH4) $(22)$ of \cite{Sharma}, (fourth-order) method (MN4) $(1.6)$ of \cite{Montazeri1} with our (third-order) method $(\ref{eqn:22})$ (M3) and  (fourth-order) method $(\ref{eqn:214})$ (M4) by the number of functional evaluations in the table 1 and by the graph in the figure 1 to figure 6.
 From figure 1 to figure 6 the colors  brown, blue, green, red  stands for the methods M3, SH4 ,  MN4 and M4. It is clear from figure 1 to figure 6 that traditional efficiency indices of fourth-order methods SH4, MN4 and M4 are same but flops-like efficiency index of our method M4 dominates other methods.

\begin{table}[htb]
\caption{Comparison of efficiency indices for different methods}
  \begin{tabular}{lllll} \hline
Iterative Methods                     &SH4         &MN4            &M3             &M4   \\ \hline 
Order of convergence                          &4           &4              &3              &4 \\    
No. of functional                       &$n+2n^2$       &$n+2n^2$        &$n+2n^2$       &$n+2n^2$ \\
evaluations   \\
Classical efficiency            &$4^{1/n+2n^2}$ &$4^{1/n+2n^2}$  &$3^{1/n+2n^2}$ &$4^{1/n+2n^2}$ \\
index                                                                                                                             \\
Number of LU                    &2            &2              &1              &1 \\
factorizations                                                                                     \\
Cost of LU factorizations                 &$4n^3/3$     &$4n^3/3$       &$2n^3/3$       &$2n^3/3$ \\
(Based on flops)                                                                                                                         \\   
Cost of linear systems         &$\frac{10n^3}{3}+2n^2$  &$\frac{7n^3}{3}+2n^2$  &$\frac{5n^3}{3}+2n^2$       &$\frac{5n^3}{3}+2n^2$ \\
(Based on flops)                                                                                                                                             \\   
Flops-like efficiency     &$4^{\frac {1}{( \frac{10n^3}{3}+4n^2+n)}} $ &$4^{\frac{1} {(\frac{7n^3}{3}+4n^2+n)}}$       &$3^{\frac{1} {(\frac{5n^3}{3}+4n^2+n)}}$       &$4^{\frac{1} {(\frac{5n^3}{3}+4n^2+n)}}$\\
index                                                                                                                                                      \\                                                    
\hline
  \end{tabular}
  \label{tab:abbr}
\end{table}

\newpage
\begin{figure}[h!]
\centering
\includegraphics[width=100mm]{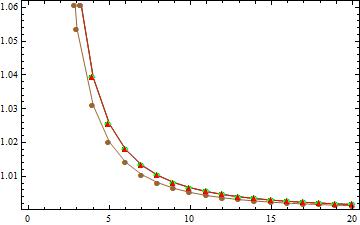}
\caption{The Plot of the (traditional) efficiency indices for different methods ( for $ n=2,3,. . .20) $ }
\label{fig:method}
\end{figure}
\begin{figure}[h!]
\centering
\includegraphics[width=100mm]{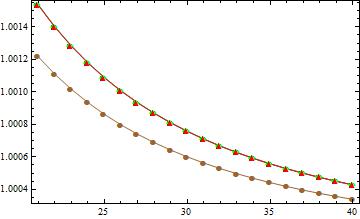}
\caption{The Plot of the (traditional) efficiency indices for different methods ( for $ n=21,22,. . .40) $ }
\label{fig:method}
\end{figure}
\begin{figure}[h!]
\centering
\includegraphics[width=100mm]{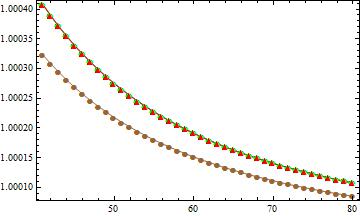}
\caption{The Plot of the (traditional) efficiency indices for different methods (for $ n=41,42,. . .80) $ }
\label{fig:method}
\end{figure}
\newpage
\begin{figure}[h!]
\centering
\includegraphics[width=100mm]{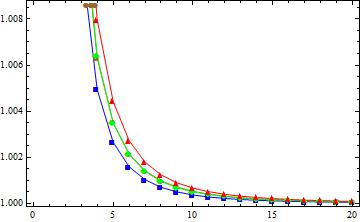}
\caption{The Plot of the flops- like efficiency indices for different methods ( for $ n=2,3,. . .20) $}
\label{fig:method}
\end{figure}
\begin{figure}[h!]
\centering
\includegraphics[width=100mm]{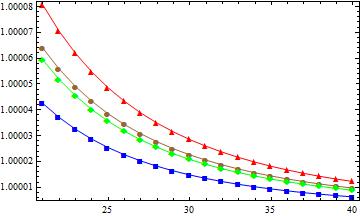}
\caption{The Plot of the flops- like efficiency indices for different methods (for $ n=21,22,. . .40) $}
\label{fig:method}
\end{figure}
\begin{figure}[h!]
\centering
\includegraphics[width=100mm]{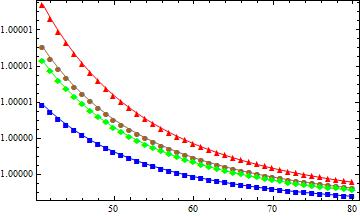}
\caption{The Plot of the flops- like efficiency indices for different methods (for $ n=41,42,. . .80) $}
\label{fig:method}
\end{figure}

 \newpage   
  
\section{Numerical Tests}
In this section, we consider our third-order method $(\ref{eqn:22})$ (M3), the fourth-order method (SH4) $(22)$ of \cite{Sharma}, the fourth-order method (MN4) $(1.6)$ of \cite{Montazeri1} and  our fourth-order method $(\ref{eqn:214})$ (M4) to compare the numerical results obtained form these methods in solving test nonlinear systems. The residual norm for first, second and third iteration is mentioned in table 2. All the computations have been done in MATHEMATICA 8. For different numerical examples the accuracy of the solutions are calculated correct up to 150 digits by using SetAccuracy command. We have used the  stopping criteria $||F(X^{(k)})|| < 1.e-150 $. We consider the following test problems:

\begin{Example}
\begin{eqnarray*}
x_1^2-x_2-19 = 0,  \nonumber \\
 -x_1^2+ \frac{x_2^3}{6}+x_2-17 = 0.
\end{eqnarray*}
\end{Example}
with initial guess $X^{(0)}=(5.1,6.1)^T$ and one of its solution is $ \alpha =(5,6)^T $. The Jacobian matrix of the above system of equations is given by: 
\[
\begin{bmatrix}
2x_1 & -1&\\
-2x_1  & 1+ \frac{x_2^2} {2} & 
\end{bmatrix}
\]

\begin{Example}
\begin{eqnarray*}
-Sin(x_1)+Cos(x_2) = 0 , \nonumber\\
 - \frac{1}{x_2}+(x_3)^{x_1} = 0, \nonumber\\
e^{x_1}-(x_3)^2=0.
\end{eqnarray*}
\end{Example}
with initial guess $X^{(0)}=(1, 0.5, 1.5)^T$ and one of its solution is $ \alpha =(0.9095...,0.6612...,1.5758...)^T $. The Jacobian matrix of the above system of equations is given by: 
\[
\begin{bmatrix}
-Cos(x_1) & -Sin(x_2)& 0 & \\
x_3^{x_1}log(x_3)& \frac{1}{x_2^2}   & x_1x_3^{-1+x_1}& \\
e^{x_1} & 0    & -2x_3 &   \\
\end{bmatrix}
\]

\begin{Example}
\begin{eqnarray*}
 x_2 x_3 +x_4 (x_2+x_3) &=& 0 \nonumber\\
x_1 x_3 + x_4 (x_1+x_3) &=& 0 \nonumber\\
x_1 x_2 + x_4 (x_1 + x_2) &=& 0 \nonumber\\
 x_1 x_2 + x_1 x_3+ x_2 x_3 &=& 1 .\nonumber
\end{eqnarray*}
\end{Example}
with initial guess $X^{(0)}=(0.5,0.5,0.5,-0.2)^T$
and one of its solution is  $ \alpha \approx $  $(0.577350,0.577350,0.577350,-0.288675)^T$ .
The Jacobian matrix of the above system of equations is given by: 
\[
\begin{bmatrix}
0 & x_3 +x_4 & x_2+ x_4 & x_2+x_3 \\
x_3+x_4   & 0 & x_1+x_4 & x_1+x_3 \\
x_2+x_4 &x_1+x_4    &0  &x_1+x_2 \\
x_2+x_3 & x_1+x_3 & x_1+x_2 & 0
\end{bmatrix}
\]
%
\newpage

\begin{Example}
\begin{eqnarray*}
-e^{x_1}+ tan^{-1} (x_2) +2 &=& 0 \nonumber\\
tan^{-1}(x_1^{2}+x_2^{2}-5) &=&0 .\nonumber
%
\end{eqnarray*}
\end{Example}
with initial guess  $ X^{(0)} =(1.0,2.0)^T$ and one of its solution is 
$ \alpha $ =$(1.12906503...1.930080863...)^T$.The Jacobian matrix of the above system of equations is given by:
\[
\begin{bmatrix}
-e^{x_1} & \frac{1}{1+x_2^2} \\ 
\frac{2x_1}{1+(5-(x_1)^2-(x_2)^2 )^2} & \frac{2x_2}{1+( 5-(x_1)^2-(x_2)^2 )^2} 
\end{bmatrix}
\]

%

\begin{Example}
\begin{eqnarray*}
-e^{-x_1}+x_2+x_3=0, \nonumber\\
-e^{-x_2}+x_1+x_3=0, \nonumber\\
-e^{-x_3}+x_1+x_2=0 .
\end{eqnarray*}
\end{Example}
with initial guess $X^{(0)}=(-0.8,1.1,1.1)^T$ and one of its solution is $ \alpha =(-0.8320...,1.1489,...,1.1489...)^T$. The Jacobian matrix of the above system of equations is given by:
\[
\begin{bmatrix}
e^{-x_1}& 1 & 1 & \\
1   & e^{-x_2} & 1&  \\
1 & 1    & e^{-x_3}  & \\
\end{bmatrix}
\]
%

\begin{Example}
\begin{eqnarray*}
x_1^2+x_2^2+x_3^2-9=0, \nonumber\\ 
x_1 x_2 x_3-1=0,  \nonumber \\
x_1+x_2-x_3^2=0 .
\end{eqnarray*}
\end{Example}
with initial guess $X^{(0)}=(3,1,2)^T$ and one of its solution is $ \alpha =(2.2242...,0.22838...,1.5837...)^T$. The Jacobian matrix of above equations is given by : 
\[
\begin{bmatrix}
2x_1 & 2x_2 & 2x_3 &  \\
x_2x_3   & x_1x_3 & x_1x_2 &  \\
1 & 1   & -2x_3 & \\
\end{bmatrix}
\]
%

\begin{Example}
\begin{eqnarray*}
log(x_2)-x_1^2+x_1 x_2=0, \nonumber\\
 log(x_1)-x_2^2+x_1 x_2=0. 
\end{eqnarray*}
\end{Example}
with initial guess $X^{(0)}=(0.5,1.5)^T$ and one of its solution is $ \alpha = (1,1)^T$. The Jacobian matrix of the above system of equations is given by :
\[
\begin{bmatrix}
-2x_1+x_2 & x_1+ \frac{1} {x_2} & \\
\frac {1} {x_1} +x_2 & x_1- 2x_2 & \\
\end{bmatrix}
\]
%
\begin{Example}
\begin{eqnarray*}
 x_i x_{i+1} -1 &=& 0  , \space \space  i= 1, 2, ...n-1 \nonumber\\
x_n x_1 -1 & = & 0 .\nonumber
\end{eqnarray*}
\end{Example}
with initial guess  $ X^{(0)} $=Table[2.0, \{i, 1, 99\}] and one of its solution is  $ \alpha   \approx  (1,1,...,1)^T $ for odd $n$. The Jacobian matrix of above system of equations is given by: 
\[
\begin{bmatrix}
x_2 &  x_1 & 0&0&...&0\\ 
0 & x_3&x_2&0&...&0 \\
0&...&0&0&x_{n-2}&x_{n-1}\\
x_n & 0 & ...& 0&0& x_1
\end{bmatrix}
\]
\begin{table}[htb]
\caption{Norm of the functions of different methods for first, second, third iteration and their flops-like efficiency indices.}
\tiny
 \begin{tabular}{|lcllll|lcl} \hline
Example & Method & $||F(x^{(1)})|| $ & $||F(x^{(2)})|| $    & $||F(x^{(3)})|| $           &E   \\ \hline
              &M3          &  8.3210e-4          &  1.9191e-13           &1.4565e-42                            &1.0357\\                                                                                    
4.1        &MN4        &7.2004e-6	       &5.2511e-27		&7.4763e-112	                     &1.0385\\						   &SH4         & 1.2923e-5          &9.2420e-26              &1.2710e-106                          &1.0315    \\	
             &M4             &2.2420e-5        & 1.4101e-24              &1.1905e-101                         &1.0452\\ \hline                                                   
            &M3               &  4.3578e-2     &  8.2464e-4                &7.4080e-9                 &1.0132\\                                                                                    
4.2           &MN4	        &	 1.1075e-2		   &1.1610e-7	&8.8842e-28                   	 &1.0137\\					&	SH4            & 1.5676e-2                 &1.1309e-6                                &2.4814e-23                &1.0108              \\	
&M4             &2.2105e-2                  & 8.8082e-6                           &1.9345e-19                                   &1.0166\\ \hline                                                     
&M3               &  5.3269e-3                 & 1.8023e-8                              &1.4083e-25              					& 1.0063\\                                                                                    
4.3&MN4	        &	 2.9921e-4		&9.1289e-17				&2.2390e-68           &1.0064	\\									
&SH4            & 5.3618e-4               &1.4537e-15                               &2.1746e-63              &1.0049               \\	
&M4             &9.3630e-4                  & 2.1533e-14                            &1.6492e-58                &1.0080\\ \hline                                                                
&M3               &  5.3521e-3                & 6.0006e-9                                &1.3577e-26             					&1.0357\\                                                                                    
4.4&MN4	        &	 1.5256e-4		 &4.0018e-18			 &3.2145e-72          	&1.0385\\								
&SH4            & 2.9895e-4               &6.5567e-17                               & 1.8332e-67                      &1.0315      \\	
&M4             &5.4871e-4                  & 9.5725e-16                           &7.0796e-63                              &1.0452\\ \hline
&M3               &  6.9918e-5                & 1.9702e-12                                &3.7793e-35             					&1.0132\\                                                                                    
4.5&MN4	        &	 9.6743e-7		 &1.8890e-24			 &4.4506e-95          	 &1.0137\\								
&SH4            & 2.1907e-6               &8.6294e-23                               & 3.6734e-87                     &1.0108       \\	
&M4             &4.2463e-6                  & 2.0792e-21                           &2.2189e-82                        &1.0166\\ \hline                                                                                                        &M3               &  4.3715e-1                & 6.3448e-4                              &2.5770e-12             					&1.0132\\                                                                                    
4.6&MN4	        &	 8.1961e-2		 &1.6321e-8			           &3.3334e-35         &1.0137 	\\								
&SH4            & 1.1046e-1               &9.6577e-8                               & 6.9429e-32               &1.0108             \\	
&M4             &1.5514e-1                  & 6.2793e-7                           &2.0478e-28                       &1.0166\\ \hline                                                         
&M3               &  4.0112e-1                & 2.3024e-2                                &6.3786e-5             					&1.0357\\                                                                                    
4.7&MN4	        &	 1.1392e-1		 &2.4875e-6			            &5.7474e-25          	 &1.0385\\								
&SH4            & 1.0359e-1               &5.4166e-6                               & 4.6302e-23                       &1.0315     \\	
&M4             &9.6796e-2                  & 9.1246e-6                           &8.2632e-22                               &1.0452\\ \hline                                                 
&M3               &  2.2955e+0                 & 1.0320e-2                              &1.3851e-9              					&1.0000 \\                                                                                    
4.8&MN4     & 6.3158 e-1		 &3.0004e-6			&2.4026e-27        &1.0000	\\						&SH4            & 7.9251 e-1              &  1.2152e-5                             &8.0715e-25              &1.0000              \\	
&M4             &1.0361e+0                  & 5.3913e-5                            &5.1123e-22               &1.0000 \\ \hline                                                                 
\end{tabular}
 \label{tab:abbr}
\end{table}

 \newpage 

\section{Conclusion}
The efficiency of quadratically multidimensional method is not satisfactory in most of the practical problem. So in this paper, first we extend third-order method of single variable to multivariate case. Since our aim is to construct the method of higher order convergence with minimum number of functional evaluations. So we have used weight concept in the same third-order method to achieve fourth-order convergence without using any  more functional evaluations. A generalized efficiency index has been discussed here. Its sense is that efficiency index is not depends on only the number of functional evaluations but also on the number of operations per iterations. We have given the comparison of efficiencies based on flops and functional evaluations.  Here it is shown that all the the traditional efficiency index of fourth-order methods are same but flops-like efficiency index of our method dominates the other methods.   The numerical results have been given to confirm validity of theoretical results. The analysis of efficiency is also connected with numerical examples.

\textsc{Anuradha Singh\\
Department of Mathematics, \\
Maulana Azad National Institute of Technology,\\
Bhopal, M.P., India-462051}.\\
E-mail: {singh.anu3366@gmail.com; singhanuradha87@gmail.com}.\\\\
\textsc{Jai Prakash Jaiswal\\
Department of Mathematics, \\
Maulana Azad National Institute of Technology,\\
Bhopal, M.P., India-462051}.\\
E-mail: {asstprofjpmanit@gmail.com; jaiprakashjaiswal@manit.ac.in}.\\\\
\end{document}